\long\def\drop#1{}
\documentclass[reqno,a4paper]{amsart}

\usepackage{mathrsfs}
\usepackage{a4wide}
\usepackage{amscd}
\usepackage{url}
\usepackage{upgreek}
\usepackage{color}
\usepackage{geompsfi-pdf}

\newtheorem{theorem}{Theorem}
\newtheorem{lemma}[theorem]{Lemma}

\def\X{\mathscr X}
\def\R{\mathbb R}
\def\N{\mathbb N}
\def\T{\mathbb T}
\def\Z{\mathbb Z}
\def\P{\mathscr P_2}
\def\bP{\mathbb P}
\def\M{\mathcal M}

\let\e\varepsilon
\let\ph\varphi
\let\weakto\rightharpoonup
\DeclareMathOperator*\argmin{argmin}
\def\div{\mathop{\mathrm{div}}}
\def\pref#1{(\ref{#1})}

\DeclareMathOperator\supp{supp}

\newenvironment{remark}%
  {\par\medbreak
    \noindent\textbf{Remark.}}%
  {\par\medskip}

\newcommand{\ssup}[1] {{{\scriptscriptstyle{({#1}})}}} 
\begin{document}

\author{Stefan Adams 
\and Nicolas Dirr 
\and Mark A. Peletier 
\and Johannes Zimmer 
}


\title[From a large-deviations principle to the Wasserstein gradient flow]{From a large-deviations principle to the Wasserstein gradient flow: a new micro-macro passage}

\date{\today}
 
\begin{abstract}
We study the connection between a system of many independent Brownian particles on one hand and the deterministic diffusion equation on the other. For a fixed time step $h>0$, a large-deviations rate functional $J_h$ characterizes the behaviour of the particle system at $t=h$ in terms of the initial distribution at $t=0$. For the diffusion equation, a single step in the time-discretized entropy-Wasserstein gradient flow is characterized by the minimization of a functional $K_h$. 
We establish a new connection between these systems by proving that $J_h$ and $K_h$ are equal up to second order in $h$ as $h\to0$.

This result gives a microscopic explanation of the origin of the entropy-Wasserstein gradient flow  formulation of the diffusion equation. Simultaneously, the limit passage presented here gives a physically natural description of the underlying particle system by describing it as an entropic gradient flow. 

\end{abstract}

\maketitle

\section{Introduction}

\subsection{Particle-to-continuum limits}

In 1905, Einstein showed~\cite{Einstein05}
 how the bombardment of a particle by surrounding fluid molecules leads to behaviour that is described by the macroscopic diffusion equation (in one dimension)
\begin{equation}
\label{eq:diffusion}
\partial_t \rho = \partial_{xx}\rho \qquad\text{for } (x,t) \in \R\times \R^+.
\end{equation}
There are now many well-established derivations of continuum equations from stochastic particle models, both formal and rigorous~\cite{DeMasiPresutti92,KipnisLandim99}.

In this paper we investigate a new method to connect some stochastic particle systems with their upscaled deterministic evolution equations, in situations where these equations can be formulated as gradient flows. This method is based on a connection between two concepts: large-deviations rate functionals associated with stochastic processes on one hand, and  gradient-flow formulations of deterministic differential equations on the other. We explain these below.

The paper is organized around a simple example: the empirical measure of a family of $n$ Brownian particles $X^{\ssup{i}}(t)\in\R, t\ge 0 $, has a limit as $n\to\infty$, which is characterized by equation~\eqref{eq:diffusion}. The natural variables to compare are the empirical measure of the position at time $ t $, i.e.  $L^t_n = n^{-1} \sum_{i=1}^n \delta_{X^{\ssup{i}}(t)}$, which describes the density of particles, and the solution $\rho(\cdot,t)$ of~\eqref{eq:diffusion}. We take a time-discrete point of view and consider time points $t=0$ and $t=h>0$. 


\subsection*{Large-deviations principles}

A large-deviations principle characterizes the fluctuation behaviour of a stochastic process. We consider the behaviour of $L^h_n$ under the condition of a given initial distribution $L^0_n\approx\rho_0\in \M_1(\R)$, where $\M_1(\R)$ is the space of probability measures on $\R$. A large-deviations result expresses the probability of finding $L_n^h$ close to some $\rho\in \M_1(\R)$ as 
\begin{equation}
\label{ldp:intro}
\mathbb{P}\bigl(L^h_n \approx \rho\,|\, L^0_n\approx \rho_0\bigr) \;\approx\; \exp \bigl[-nJ_h(\rho\,;\rho_0)\bigr]\qquad\text{as }n\to\infty.
\end{equation}
The functional $J_h$ is called the \emph{rate function}. By~\eqref{ldp:intro}, $J_h(\rho\,;\rho_0)$ characterizes the probability of observing a given realization $\rho$: large values of $J_h$ imply small probability. Rigorous statements are given below.

\subsection*{Gradient flow-formulations of parabolic PDEs} An equation such as~\eqref{eq:diffusion}  characterizes an evolution in a state space $\X$, which in this case we can take as $\X = \M_1(\R)$ or $\X=L^1(\R)$. A \emph{gradient-flow formulation} of the equation is an equivalent formulation with a specific structure. It employs two quantities, a functional $E\colon\X\to\R$ and a \emph{dissipation metric} $d\colon\X\times \X\to\R$. Equation~\eqref{eq:diffusion} can be written as the gradient flow of the entropy functional $E(\rho) = \int \rho\log\rho\, dx$ with respect to the Wasserstein metric $d$ (again, see below for precise statements). 
We shall use the following property: the solution $t\mapsto \rho(t, \cdot)$ of~\eqref{eq:diffusion} can be approximated by the time-discrete sequence $\{\rho^n\}$ defined recursively by 
\begin{equation}
  \label{def:BackwardEuler}
  \rho^n \in \argmin_{\rho\in\X} K_h(\rho\;;\rho^{n-1}), \qquad 
  K_h(\rho\;;\rho^{n-1}) :=\frac1{2h} d(\rho,\rho^{n-1})^2 + E(\rho) - E(\rho^{n-1}).
\end{equation}

\subsection*{Connecting large deviations with gradient flows}

The results of this paper are illustrated in the diagram below. 
\makeatletter
\atdef@ O#1O#2O{\CD@check{O..O..O}{\llap{$\m@th\vcenter{\hbox 
 {$\scriptstyle#1$}}$}\Big\updownarrow 
 \rlap{$\m@th\vcenter{\hbox{$\scriptstyle#2$}}$}&&}} 
\makeatother
\def\ntext#1{\text{\normalsize #1}}
\begin{equation}
\label{diag:commuting}
\begin{CD} 
\substack{\textstyle\text{discrete-time}\\[\jot]\ntext{rate functional }\textstyle J_h}@>\text{this paper}>\substack{\text{Gamma-convergence}\\h\to0}
> \substack{\textstyle\text{discrete-time variational}\\[\jot]\textstyle\text{formulation }K_h}\\ 
@O\substack{\text{large-deviations principle}\\n\to\infty}OO 
  @OO{h\to0}O\\ 
\text{Brownian particle system} @>\text{continuum limit}>n\to\infty> \text{continuum equation~\pref{eq:diffusion}} 
\end{CD}
\end{equation}
The lower level of this diagram is the classical connection: in the limit $n\to\infty$, the empirical measure $t\mapsto L_n^t$ converges to the solution $\rho$ of equation~\pref{eq:diffusion}. In the left-hand column the large-deviations principle mentioned above connects the particle system with the rate functional $J_h$. The right-hand column is the formulation of equation~\pref{eq:diffusion} as a gradient flow, in the sense that the time-discrete approximations constructed by successive minimization of $K_h$ converge to~\pref{eq:diffusion} as $h\to0$. 

Both functionals $J_h$ and $K_h$ describe a single time step of length $h$: $J_h$ characterizes the fluctuations of the particle system after time $h$, and $K_h$ characterizes a single time step of length~$h$ in the time-discrete approximation of~\eqref{eq:diffusion}.  In this paper we make a new connection, a Gamma-convergence result relating $J_h$ to $K_h$, indicated by the top arrow. It is this last connection that is the main mathematical result of this paper.

This result is interesting for a number of reasons. First, it places the entropy-Wasserstein gradient-flow formulation of~\eqref{eq:diffusion} in the context of large deviations for a system of Brownian particles. In this sense it gives a microscopic justification of the coupling between the entropy functional and the Wasserstein metric, as it occurs in~\eqref{def:BackwardEuler}. Secondly, it shows that $K_h$ not only characterizes the deterministic evolution via its minimizer, but also the fluctuation behaviour via the connection to $J_h$. Finally, it suggests a principle that may be much more widely valid, in which gradient-flow formulations have an intimate connection with large-deviations rate functionals associated with stochastic particle systems.

\medskip

The structure of this paper is as follows. We first introduce the specific system of this paper and formulate the existing large-deviations result~\eqref{ldp:intro}. In Section~\ref{sec:gf} we discuss the abstract gradient-flow structure and recall the definition of the Wasserstein metric. Section~\ref{sec:central-statement} gives the central result,  and Section~\ref{sec:Discussion} provides a discussion of the background and relevance. Finally the two parts of the proof of the main result, the upper and lower bounds, are given in Sections~\ref{sec:upperbound} and~\ref{sec:lowerbound}.

Throughout this paper, measure-theoretical notions such as absolute continuity are with respect to the Lebesgue measure, unless indicated otherwise. By abuse of notation, we will often identify a measure with its Lebesgue density.

\section{Microscopic model and Large-Deviations Principle}
\label{sec:Micr-model-Large}


Equation~\pref{eq:diffusion} arises as the hydrodynamic limit of a wide variety of particle systems. In this paper we consider the simplest of these, which is a collection of $n$ independently moving \emph{Brownian particles}. A Brownian particle is a particle whose position in $\R$ is given by a Wiener process, for which the probability of a particle moving from $x\in\R$ to $y\in \R$ in time $h>0$ is given by the probability density
\begin{equation}
  \label{def:P}
  p_h(x,y) := \frac1{(4\pi h)^{1/2}} {e}^{-(y-x)^2/4h}.
\end{equation}
Alternatively, this corresponds to the Brownian bridge measure for the $n$ random elements in the space of all continuous functions  $ [0,h]\mapsto \R $. We work with Brownian motions having generator $ \Delta $ instead of $ \frac{1}{2}\Delta $, and we write $ \mathbb{P}_x $ for the probability measure under which $ X=X^{\ssup{1}} $ starts from $ x\in\R $. 

We now specify our  system of Brownian particles. Fix a measure $\rho_0\in \M_1(\R)$ which will serve as the initial distribution of the $ n $ Brownian motions $ X^{\ssup{1}},\ldots, X^{\ssup{n}} $ in $ \R $. 
For each $n\in \N$, we let $(X^{\ssup{i}})_{i=1,\ldots,n} $ be a collection of independent Brownian motions, whose distribution is given by the product $ \bP_n=\bigotimes_{i=1}^n\bP_{\rho_0} $, where $ \bP_{\rho_0}=\rho_0(d x)\bP_x $ is the probability measure under which $ X=X^{\ssup{1}} $ starts with initial distribution $ \rho_0 $.


It follows from the definition of the Wiener process and the law of large numbers that the {empirical measure} $L^t_n$, the random probability measure in  $\M_1(\R)$ defined by
\[
L_n^t :=\frac1n \sum_{i=1}^n \delta_{X^{\ssup{i}}(t)},
\]
converges in probability to the solution $\rho$ of~\pref{eq:diffusion} with initial datum $\rho_0$. In this sense the equation~\pref{eq:diffusion} is the many-particle limit of the Brownian-particle system. Here and in the rest of this paper the convergence $\weakto$ is the weak-$\ast$ or weak convergence for probability measures, defined by the duality with the set of continuous and bounded functions $C_b(\R)$.


Large-deviations principles are given for many empirical measures of the $ n $ Brownian motions under the product measure $ \bP_n $. Of particular interest to us is  the empirical measure for the pair of the initial and terminal position for a given time horizon $ [0,h] $, that is, the empirical pair measure
$$
Y_n=\frac{1}{n}\sum_{i=1}^n\delta_{(X^{\ssup{i}}(0),X^{\ssup{i}}(h))}.
$$ 
Note that the empirical measures $ L_n^0 $ and $ L_n^h $ are the first and second marginals of $ Y_n $.

The \emph{relative entropy} $H\colon\M_1(\R\times\R)^2\to[0,\infty]$ is the functional
\[
H(q\,|\,p):=\begin{cases} \int\limits_{\R\times\R} f(x,y)\log f(x,y) \;p(d (x,y)) & \mbox{ if } q\ll p, f=\frac{d q}{d p}\\ + \infty & \mbox{ otherwise.}\end{cases}
\]
For given $ \rho_0,\rho\in\M_1(\R) $ denote by 
\begin{equation}
\label{def:Gamma}
\Gamma(\rho_0,\rho)=\{q\in\M_1(\R\times\R)\colon \pi_0q=\rho_0,\pi_1q=\rho\}
\end{equation}
the set of pair measures whose first marginal $ \pi_0 q(d \cdot):=\int_\R q(d\cdot, d y) $ equals $ \rho_0 $ and whose second marginal $ \pi_1q(d \cdot):=\int_\R q(d x, d\cdot)$ equals $ \rho $. 
For a given $ \delta>0 $ we denote by $ B_\delta=B_\delta(\rho_0) $ the open ball with radius $ \delta>0 $ around $ \rho_0 $ with respect to the L\'{e}vy metric on $ \M_1(\R) $~\cite[Sec.~3.2]{DeuschelStroock89}.

\begin{theorem}[\textbf{Conditional large deviations}]
\label{th:LDP1}
Fix $\delta>0$ and $\rho_0\in\M_1(\R)$. 
The sequence $ (\bP_n\circ (L_n^h)^{-1})_{n\in\N} $ satisfies under the condition that $ L_n^0\in B_\delta(\rho_0)$ a large deviations principle on $ \M_1(\R) $ with rate $ n $ and rate function
\begin{equation}
\label{def:Jh}
J_{h,\delta}(\rho\,;\rho_0) := \inf_{q\colon\pi_0 q\in B_\delta(\rho_0), \pi_1q=\rho} H(q\,|\,q_0),\quad \rho\in\M_1(\R),
\end{equation}
where 
\begin{equation}
\label{def:q_0}
q_0(d x, d y):=\rho_0(d x)p_h(x,y)d y.
\end{equation}
This means that
\begin{enumerate}
\item For each open $O\subset \M_1(\R)$, 
\[
\liminf_{n\to\infty} \frac1n \log \bP_n\bigl(L_n^h\in O \,|\, L_n^0\in B_\delta(\rho_0)\bigr) 
\geq -\inf_{\rho\in O} J_{h,\delta}(\rho\,;\,\rho_0).
\]
\item For each closed $K\subset \M_1(\R)$, 
\[
\limsup_{n\to\infty} \frac1n \log \bP_n\bigl(L_n^h\in K\,|\, L_n^0\in B_\delta(\rho_0)\bigr) 
\leq -\inf_{\rho\in K} J_{h,\delta}(\rho\,;\,\rho_0).
\]
\end{enumerate}
\end{theorem}

A proof of this standard result can be given by an argument along the following lines. First, note that
$$
\bP_{\rho_0}\circ(\sigma_0,\sigma_h)^{-1}(x,y)=\rho_0(d x)\bP_x(X(h)\in d y)/d y=\rho_0(d x)p_h(x,y)d y=:q_0(d x, dy),\quad x,y\in\R, 
$$ 
where $ \sigma_s\colon C([0,h];\R) \to\R, \omega\mapsto \omega(s) $ is the projection of any path $ \omega $ to its position at time $ s\ge 0 $. By Sanov's Theorem, the sequence  $ (\bP_n\circ Y_n^{-1})_{n\in\N} $ of the empirical pair measures $ Y_n $  satisfies a large-deviations principle on $ \M_1(\R\times\R) $ with speed  $ n $ and rate function $ q\mapsto H(q\,|\,q_0), q\in\M_1(\R\times\R) $, see e.g. \cite{denHollander00,Csiszar84}). Secondly, the contraction principle (e.g.,~\cite[Sec. III.5]{denHollander00}) shows that the pair of marginals $ (L_n^0,L_n^h) =(\pi_0Y_n,\pi_1 Y_n)$ of $ Y_n $ satisfies a large deviations principle on $ \M_1(\R)\times\M_1(\R) $ with rate $ n $ and rate function
$$
(\tilde \rho_0,\rho) \mapsto 
\inf_{q\in\M_1(\R\times\R)\colon \pi_0q=\tilde \rho_0,\pi_1q=\rho} H(q\,|\,q_0),
$$
for any $\tilde \rho_0,\rho\in\M_1(\R)$. Thirdly, as in the first step, it follows that the empirical measure $ L_n^0 $ under $ \bP_n $ satisfies a large deviations principle on $ \M_1(\R) $ with rate $ n $ and rate function $\tilde \rho_0 \mapsto  H(\tilde \rho_0\,|\,\rho_0) $, for $\tilde \rho_0\in\M_1(\R) $.

Therefore for a subset $A\subset\M_1(\R)$, 
\begin{align*}
\frac1n \log \bP_n(L_n^h\in A \,|\, L_n^0\in B_\delta)
&= \frac1n \log \bP_n(L_n^h\in A,  L_n^0\in B_\delta) - \frac1n \log \bP_n( L_n^0\in B_\delta)\\
&\sim \inf_{q\colon \pi_0q\in B_\delta, \pi_1q\in A} H(q\,|\,q_0)
- \inf_{\tilde \rho_0\in B_\delta} H(\tilde \rho_0\,|\,\rho_0).
\end{align*}
Since $\rho_0\in B_\delta$, the latter infimum equals zero, and the claim of Theorem~\ref{th:LDP1} follows.

\medskip

We now consider the limit of the rate functional as the radius $\delta\to 0$. Two notions of convergence are appropriate, that of pointwise convergence and Gamma convergence. 

\begin{lemma}
\label{lemma:Gamma-conv-Jhd}
Fix $\rho_0\in \M_1(\R)$. As $\delta\downarrow0$, $J_{h,\delta}(\,\cdot\,;\rho_0)$ converges in $\M_1(\R)$ both in the pointwise and in the Gamma sense  to 
\[
J_h(\rho\,;\,\rho_0) := \inf_{q\colon \pi_0q=\rho_0, \pi_1q=\rho} H(q\,|\,q_0).
\]
Gamma convergence means here that
\begin{enumerate}
\item \label{item:gamma:lowerbound} 
(Lower bound)
For each sequence $\rho^\delta \weakto \rho$ in $\M_1(\R)$,
\begin{equation}
\label{th:gamma:lowerbound}
\liminf_{\delta\to0}   J_{h,\delta}(\rho^\delta;\rho_0) \geq  J_h(\rho\,;\,\rho_0) ,
\end{equation}
\item \label{item:gamma:recovery}
(Recovery sequence)
For each $\rho\in \M_1(\R)$, there exists a sequence $(\rho^\delta)\subset \M_1(\R)$ with $\rho^\delta\weakto \rho$ such that
\begin{equation}
\label{th:gamma:recovery}
\lim_{\delta\to0} J_{h,\delta}(\rho^\delta\;;\rho_0)  =  J_h(\rho\,;\,\rho_0).
\end{equation}
\end{enumerate}

\end{lemma}

\begin{proof}
$J_{h,\delta}(\,\cdot\,;\rho_0)$ is an increasing  sequence of convex functionals on $\M_1(\R)$; therefore it converges at each fixed $\rho\in\M_1(\R)$. The Gamma-convergence then follows from, e.g.,~\cite[Prop.~5.4]{DalMaso93} or~\cite[Rem.~1.40]{Braides02}.
\end{proof}

\begin{remark}
L\'eonard~\cite{Leonard07TR} proves a similar statement, where he replaces the ball $B_\delta(\rho_0)$ in Theorem~\ref{th:LDP1} by an explicit sequence $\rho_{0,n}\weakto \rho_0$. The rate functional that he obtains is again $J_h$.
\end{remark}

Summarizing, the combination of Theorem~\ref{th:LDP1} and Lemma~\ref{lemma:Gamma-conv-Jhd} forms a rigorous version of the statement~\eqref{ldp:intro}. The parameter $\delta$ in Theorem~\ref{th:LDP1} should be thought of as an artificial parameter, introduced to make the large-deviations statement non-singular, and which is eliminated by the  Gamma-limit of Lemma~\ref{lemma:Gamma-conv-Jhd}.
%
\section{Gradient flows}
\label{sec:gf}

Let us briefly recall the concept of a gradient flow, starting with flows in $\R^d$.
The gradient flow in $\R^d$ of a functional $E\colon\R^d\to\R$ is the evolution in $\R^d$ given by 
\begin{equation}
\label{def:gf-rn}
\dot x^i(t) = -\partial_{i} E(x(t))
\end{equation}
which can be written in a geometrically more correct way as 
\begin{equation}
\label{def:gf-rm}
\dot x^i(t) = -g^{ij} \partial_{j} E(x(t)).
\end{equation}
The metric tensor $g$ converts the covector field $\nabla E$ into a vector field that can be assigned to $\dot x$. In the case of~\eqref{def:gf-rn} we have $g^{ij} = \delta^{ij}$, the Euclidean metric, and for a general Riemannian manifold with metric tensor $g$, equation~\eqref{def:gf-rm} defines the gradient flow of $E$ with respect to $g$. 

In recent years this concept has been generalized to general metric spaces~\cite{AmbrosioGigliSavare05}. This generalization is partly driven by the fact, first observed by Jordan, Kinderlehrer, and Otto~\cite{JordanKinderlehrerOtto97,JordanKinderlehrerOtto98}, that many parabolic evolution equations of a diffusive type can be written as gradient flows in a space of measures with respect to the Wasserstein metric. 
The Wasserstein distance is defined on the set of probability measures with finite second moments, 
\[
\P(\R) := \left\{\rho\in\M_1(\R): \int_{\R} x^2 \,\rho(dx)<\infty\right\},
\]
and is given by
\begin{equation}
\label{def:d}
d(\rho_0,\rho_1)^2 := \inf_{\gamma\in\Gamma(\rho_0,\rho_1)} \int_{\R\times\R} (x-y)^2\, \gamma(d(x,y)),
\end{equation}
where $\Gamma(\rho_0,\rho_1)$ is defined in~\eqref{def:Gamma}. 

Examples of parabolic equations that can be written as a gradient flow of some energy $E$ with respect to the Wasserstein distance are
\begin{itemize}
\item The diffusion equation~\eqref{eq:diffusion}; this is the gradient flow of the (negative) \emph{entropy}
\begin{equation}
\label{def:entropy}
E(\rho) := \int_{\R} \rho\log \rho\, dx ;
\end{equation}
\item nonlocal convection-diffusion equations~\cite{JordanKinderlehrerOtto98,AmbrosioGigliSavare05,CarrilloMcCannVillani06} of the form
\begin{equation}
\label{ex:nonlocalconvdiff}
\partial_t \rho = \div \rho\nabla \bigl[U'(\rho) + V + W*\rho\bigr],
\end{equation}
where $U$, $V$, and $W$ are given functions on $\R$, $\R^d$, and $\R^d$, respectively;
\item higher-order parabolic equations~\cite{Otto98a,GiacomelliOtto01,Glasner03,MatthesMcCannSavare09TR,GianazzaSavareToscani08} of the form
\begin{equation}
\label{ex:higher-order}
\partial_t \rho =- \div \rho\nabla\bigl(\rho^{\alpha-1}\Delta\rho^\alpha\bigr), 
\end{equation}
for $1/2\leq \alpha\leq 1$;
\item moving-boundary problems, such as a prescribed-angle lubrication-approximation model~\cite{Otto98a}
\begin{equation}
\label{ex:lubrication}
\begin{aligned}
&\partial_t \rho = -\partial_x (\rho \,\partial_{xxx}\rho) &\qquad&\text{in }\{\rho>0\}\\
&\partial_x\rho = \pm 1 && \text{on }\partial \{\rho>0\},
\end{aligned}
\end{equation}
and a model of crystal dissolution and precipitation~\cite{PortegiesPeletier08TR}
\begin{equation}
\label{ex:Portegies}
\partial_t \rho = \partial_{xx} \rho \text{ in }\{\rho>0\}, \qquad\text{ with }\qquad\partial_n \rho = -\rho v_n\text{ and }v_n=f(\rho)\text{ on }\partial\{\rho>0\}.
\end{equation}
\end{itemize}

\section{The central statement}
\label{sec:central-statement}

The aim of this paper is to connect $J_h$ to the functional $K_h$ in the limit $h\to0$, in the sense that 
\begin{equation}
\label{formalresult}
J_h(\;\cdot\;;\rho_0) \sim \frac12 K_h(\;\cdot\;;\rho_0)\qquad \text{as }h\to0.
\end{equation}
For any $\rho\not=\rho_0$ both $J_h(\rho\,;\,\rho_0)$ and $K_h(\rho\,;\,\rho_0)$ diverge as $h\to0$, however, and we therefore reformulate this statement in the form
\[
J_h(\;\cdot\;;\rho_0) -\frac1{4h} d(\,\cdot\,,\rho_0)^2 \longrightarrow\frac12 E(\;\cdot\;) - \frac12 E(\rho_0).
\]

The precise statement is given in the theorem below. This theorem is probably true in greater generality, possibly even for all $\rho_0,\rho\in \P(\R^d)$. For technical reasons we need to impose restrictive conditions on $\rho_0$ and $\rho$, and to work in one space dimension, on a bounded domain $[0,L]$.

For any $0<\delta<1$ we define the set
\[
A_\delta := \left\{\rho\in L^\infty(0,L): \int_0^L \rho = 1\text{ and }\|\rho-L^{-1}\|_\infty < \delta\right\}.
\]

\begin{theorem}
\label{th:main}
Let $J_h$ be defined as in~\pref{def:Jh}. Fix $L>0$; there exists $\delta>0$ with the following property. 

Let $\rho_0\in A_\delta\cap C([0,L])$. Then
\begin{equation}
\label{conv:Gamma}
J_h(\,\cdot\;;\rho_0) - \frac1{4h}d(\,\cdot\,,\rho_0)^2 \longrightarrow 
\frac12 E(\cdot) - \frac12 E(\rho_0)
\qquad\text{as }h\to0,
\end{equation}
in the set $A_\delta$, 
where the arrow denotes Gamma-convergence with respect to the narrow topology. 
In this context this means that the two following conditions hold:
\begin{enumerate}
\item \label{item:th:main:lowerbound}
(Lower bound)
For each sequence $\rho^h\weakto \rho$ in $A_\delta$, 
\begin{equation}
\label{th:main:lowerbound}
\liminf_{h\to0} J_h(\rho^h\;;\rho_0) - \frac1{4h}d(\rho^h,\rho_0)^2\geq  \frac12 E(\rho) - \frac12 E(\rho_0).
\end{equation}
\item \label{item:th:main:recovery}
(Recovery sequence)
For each $\rho\in A_\delta$, there exists a sequence $(\rho^h)\subset A_\delta$ with $\rho^h\weakto \rho$ such that
\begin{equation}
\label{th:main:recovery}
\lim_{h\to0} J_h(\rho^h\;;\rho_0) - \frac1{4h}d(\rho^h,\rho_0)^2 =  \frac12 E(\rho) - \frac12 E(\rho_0).
\end{equation}
\end{enumerate}
\end{theorem}

\section{Discussion}
\label{sec:Discussion}

There are various ways to interpret Theorem~\ref{th:main}.

\smallskip
\emph{An explanation of the functional $K_h$ and the minimization problem~\eqref{def:BackwardEuler}.}
The authors of~\cite{JordanKinderlehrerOtto98} motivate the minimization problem~\eqref{def:BackwardEuler} by analogy with the well-known backward Euler approximation scheme. Theorem~\ref{th:main} provides an independent explanation of this minimization problem, as follows. By the combination of~\pref{ldp:intro} and~\pref{formalresult}, the value $K_h(\rho\,;\rho_0)$ determines the probability of observing $\rho$ at time~$h$, given a distribution $\rho_0$ at time zero. Since for large $n$ only near-minimal values of $J_h$, and therefore of $K_h$, have non-vanishing probability, this explains why the minimizers of $K_h$ arise. It also shows that the minimization problem~\eqref{def:BackwardEuler}, and specifically the combination of the entropy and the Wasserstein terms, is not just a mathematical construct but also carries physical meaning.

A related interpretation stems from the fact that~\pref{ldp:intro} characterizes not only the most probable state, but also the fluctuations around that state. Therefore $J_h$ and by~\pref{formalresult} also $K_h$ not only carry meaning in their respective minimizers, but also in the behaviour away from the minimum. Put succinctly: $K_h$ also characterizes the fluctuation behaviour of the particle system, for large but finite $n$.

\smallskip
\emph{A microscopic explanation of the entropy-Wasserstein gradient flow.} 
The diffusion equation~\eqref{eq:diffusion} is a gradient flow in many ways simultaneously: it is the gradient flow of the Dirichlet integral $\frac12\int |\nabla \rho|^2$ with respect to the $L^2$ metric, of $\frac12 \int \rho^2$ with respect to the $H^{-1}$ metric; more generally, of the $H^s$ semi-norm with respect to the $H^{s-1}$ metric. In addition there is of course the gradient flow of the entropy $E$ with respect to the Wasserstein metric.

Theorem~\eqref{th:main} shows that among these the entropy-Wasserstein combination is special, in the sense that it not only captures the deterministic limit, i.e., equation~\pref{eq:diffusion}, but also the fluctuation behaviour at large but finite $n$. Other gradient flows may also produce~\pref{eq:diffusion}, but they will not capture the fluctuations, \emph{for this specific stochastic system}. Of course, there may be other stochastic particle systems for which not the entropy-Wasserstein combination but another combination reproduces the fluctuation behaviour.

\smallskip

There is another way to motivate the combination of entropy and the Wasserstein distance. In~\cite{KipnisOlla90} the authors derive a rate functional for the time-continuous problem, which is therefore a functional on a space of space-time functions such as $C(0,\infty;L^1(\R^d))$. The relevant term for this discussion is 
\[
I(\rho) := \inf_v \left\{\int_0^\infty\int_{\R^d} |v(x,t)|^2 \rho(x,t)\, dxdt: 
  \partial_t\rho = \Delta \rho + \div \rho v \right\}.
\]
If we rewrite this infimum by $v= w - \nabla \rho$ instead as
\[
\inf_w \left\{\int_0^\infty\int_{\R^d} |w(x,t)-\nabla(\log \rho + 1) |^2 \rho(x,t)\, dxdt: 
  \partial_t\rho = \div \rho w \right\},
\]
then we recognize that this expression penalizes deviation of $w$ from the variational derivative (or $L^2$-gradient) $\log\rho + 1$ of $E$. Since the expression $\int_{\R^d} |v|^2\rho \, dx$ can be interpreted as the derivative of the Wasserstein distance (see~\cite{Otto01} and~\cite[Ch.~8]{AmbrosioGigliSavare05}), this provides again a connection between the entropy and the Wasserstein distance. 

\smallskip

\emph{The origin of the Wasserstein distance.}
The proof of Theorem~\ref{th:main} also allows us to trace back the origin of the Wasserstein distance in the limiting functional $K_h$. It is useful to compare $J_h$ and $K_h$ in a slightly different form. Namely,  using~\pref{def:d} and the expression of $H$ introduced in~\pref{expansion:H} below, we write
\begin{align}
  \label{expansion:Jh}
  J_h(\rho\,;\rho_0) &= \inf_{q\in \Gamma(\rho_0,\rho)} \Biggl\{E(q) -   E(\rho_0) + \log 2\sqrt{\pi h} + \frac1{4h} \iint\limits_{\R\times\R} (x-y)^2 q(x,y)\, dxdy\Biggr\},\\
\frac12 K_h(\rho\,;\rho_0) &= \frac12 E(\rho) - \frac12 E(\rho_0) 
+ \frac1{4h} \inf_{q\in \Gamma(\rho_0,\rho)} \iint\limits_{\R\times\R} (x-y)^2 q(x,y)\, dxdy.\notag
\end{align}
One similarity between these expressions is the form of the last term in both lines, combined with the minimization over $q$. Since that last term is prefixed by the large factor $1/4h$, one expects it to dominate the minimization for small $h$, which is consistent with the passage from the first to the second line. 

In this way the Wasserstein distance in $K_h$ arises from the last term in~\pref{expansion:Jh}. Tracing back the origin of that term, we find that it originates in the exponent $(x-y)^2/4h$ in $P^h$ (see~\pref{def:P}), which itself arises from the Central Limit Theorem. In this sense the Wasserstein distance arises from the same Central Limit Theorem that provides the properties of Brownian motion in the first place. 

This also explains, for instance, why we find the Wasserstein distance of order $2$ instead of any of the other orders. This observation also raises the question whether stochastic systems with heavy-tail behaviour, such as observed in fracture networks~\cite{BerkowitzScher98, BerkowitzScherSilliman00} or near the glass transition~\cite{WeeksWeitz02}, would be characterized by a different gradient-flow structure. 

\smallskip
\emph{A macroscopic description of the particle system as an entropic gradient flow.} For the simple particle system under consideration, the macroscopic description by means of the diffusion equation is well known; the equivalent description as an entropic gradient flow is physically natural, but much more recent. The method presented in this paper is a way to obtain this entropic gradient flow directly as the macroscopic description, without having to consider solutions of the diffusion equation. This rigorous passage to a physically natural macroscopic limit may lead to a deeper understanding of particle systems, in particular in situations where the gradient flow formulation is mathematically more tractable.


\smallskip
\emph{Future work.}
Besides the natural question of generalizing Theorem~\ref{th:main} to a larger class of probability measures, including measures in higher dimensions, there are various other interesting avenues of investigation. A first class of extensions is suggested by the many differential equations that can be written in terms of Wasserstein gradient flows, as explained in Section~\ref{sec:gf}: can these also be related to large-deviation principles for well-chosen stochastic particle systems? Note that many of these equations correspond to systems of \emph{interacting} particles, and therefore the large-deviation result of this paper will need to be generalized.

Further extensions follow from relaxing the assumptions on the Brownian motion. \emph{Kramers' equation}, for instance, describes the motion of particles that perform a Brownian motion in \emph{velocity} space, with the position variable following deterministically from the velocity. The characterization by Huang and Jordan~\cite{Huang00,HuangJordan00} of this equation as a gradient flow with respect to a modifed Wasserstein metric suggests a similar connection between gradient-flow and large-deviations structure.

\section{Outline of the arguments}

Since most of the appearances of $h$ are combined with a factor $4$, it is notationally useful to incorporate the $4$ into it. We do this by introducing the new small parameter
\[
\e^2 := 4h,
\]
and we redefine the functional of equation~\eqref{def:BackwardEuler},
\[
\frac12 K_\e(\rho\;;\rho_0) := \frac1{\e^2} d(\rho,\rho_0)^2 + \frac12E(\rho) - \frac12E(\rho_0),
\]
and analogously for~\eqref{def:Jh}
\begin{equation}
\label{def:Je}
J_\e(\rho\;;\rho_0) := \inf_{q\in \Gamma(\rho_0,\rho)} H(q\,|\,q_0),
\end{equation}
where $q_0(dxdy) = \rho_0(dx)p_\e(x,y)dy$, with \[
p_\e(x,y) := \frac1{\e\sqrt\pi} e^{-(y-x)^2/\e^2},
\]
in analogy to~\eqref{def:P} and~\eqref{def:q_0}.
Note that 
\begin{align}
\label{expansion:H}
H(q\,|\,q_0) &= E(q) - \iint_{\R\times\R} q(x,y)\log\bigl[\rho_0(x)p_\e(x,y)\bigr]\, dxdy \notag\\
&= E(q) - E(\rho_0) + \frac12 \log \e^2\pi 
  +\frac1{\e^2} \iint\limits_{\R\times\R} (x-y)^2 q(x,y)\, dxdy,
\end{align}
where we abuse notation and write $E(q) = \int_{\R\times\R} q(x,y)\log q(x,y)\, dxdy$. 

\subsection{Properties of the Wasserstein distance}

We now discuss a few known properties of the Wasserstein distance.
\begin{lemma}[Kantorovich dual formulation~\cite{Villani03,AmbrosioGigliSavare05,Villani08}]
\label{lemma:phph}
Let $\rho_0,\rho_1\in \P(\R)$ be absolutely continuous with respect to Lebesgue measure. Then 
\begin{equation}
\label{def:phph*}
d(\rho_0,\rho_1)^2 = \sup_\ph \left\{\int_{\R} (x^2 - 2\ph(x))\rho_0(x)\, dx
 + \int_{\R} (y^2-2\ph^*(y))\rho_1(y)\, dy: \; \ph\colon\R\to\R \text{ convex }\right\},
\end{equation}
where $\ph^*$ is the convex conjugate (Legendre-Fenchel transform) of $\ph$, and where the supremum is achieved.
In addition, at $\rho_0$-a.e. $x$ the optimal function $\ph$ is twice differentiable, and 
\begin{equation}
\label{eq:phpp}
\ph''(x) = \frac{\rho_0(x)}{\rho_1(\ph'(x))}.
\end{equation}
A similar statement holds for $\ph^*$, 
\begin{equation}
\label{eq:phpptwo}
(\ph^*)''(y) = \frac{\rho_1(y)}{\rho_0((\ph^*)'(y))}.
\end{equation}
\end{lemma}

For an absolutely continuous $q\in\P(\R\times\R)$ we will often use the notation 
\[
d(q)^2 := \iint\limits _{\R\times\R} (x-y)^2 \, q(x,y)\, dxdy.
\]
Note that 
\[
d(\rho_0,\rho_1) = \inf\{ d(q): \pi_{0,1}q = \rho_{0,1}\},
\]
and that if $\pi_{0,1}q=\rho_{0,1}$, and if the convex functions $\ph$, $\ph^*$ are associated with $d(\rho_0,\rho_1)$ as above, then the difference can be expressed as
\begin{align}
d(q)^2 - d(\rho_0,\rho_1)^2 
&=\begin{aligned}[t]
\iint\limits_{\R\times\R} (x-y)^2 \, q(x,y)\, dxdy
- \iint\limits_{\R\times\R} &(x^2- 2\ph(x))\,q(x,y)\, dxdy
 \\
&- \iint\limits_{\R\times\R} (y^2-2\ph^*(y))\,q(x,y)\, dxdy\end{aligned}\notag\\
&= 2\iint\limits_{\R\times\R} (\ph(x)+\ph^*(y) - xy)\,q(x,y)\, dxdy.
\label{equiv:dq}
\end{align}

\subsection{Pair measures and $\tilde q_\e$}

A central role is played by the following, explicit measure in $\P(\R\times\R)$. For given $\rho_0\in \M_1(\R)$ and a sequence of absolutely continuous measures $\rho^\e\in \M_1(\R)$, we define the absolutely continuous measure $\tilde q^\e\in\M_1(\R\times\R)$ by 
\begin{equation}
\label{def:tildeq}
\tilde q^\e(x,y) := Z_\e^{-1} \frac1{\e\sqrt\pi}\sqrt{\rho_0(x)}\sqrt{\rho^\e(y)}
  \exp\Bigl[\frac2{\e^2}(xy-\ph_\e(x) - \ph_\e^*(y))\Bigr],
\end{equation}
where the normalization constant $Z_\e$ is defined as
\begin{equation}
\label{def:Ze}
Z_\e = Z_\e(\rho_0,\rho^\e):= \frac1{\e\sqrt\pi}\iint\limits_{\R\times\R} \sqrt{\rho_0(x)}\sqrt{\rho^\e(y)}
  \exp\Bigl[\frac2{\e^2}(xy-\ph_\e(x) - \ph_\e^*(y))\Bigr]\, dxdy.
\end{equation}
In these expressions, the functions $\ph_\e$, $\ph_\e^*$ are associated with $d(\rho_0,\rho^\e)$ as by Lemma~\ref{lemma:phph}. 
Note that the marginals of $\tilde q^\e$ are not equal to $\rho_0$ and $\rho^\e$, but they do converge (see the proof of part~\ref{item:th:main:recovery} of Theorem~\ref{th:main}) to $\rho_0$ and the limit $\rho$ of $\rho^\e$.

\subsection{Properties of $\tilde q^\e$ and $Z_\e$}
\label{subsec:overview}

The role of $\tilde q^\e$ can best be explained by the following observations. 
We first discuss the \emph{lower bound}, part~\ref{item:th:main:lowerbound} of Theorem~\ref{th:main}.
If~$q^\e$ is optimal in the definition of $J_\e(\rho^\e\,;\,\rho_0)$---implying that it has marginals $\rho_0$ and $\rho^\e$---then
\begin{align}
0&\leq H(q^\e|\tilde q^\e)
= E(q^\e) - \iint q^\e\log\tilde q^\e
\notag\\
&= E(q^\e) + \log Z_\e +\frac12 \log \e^2\pi - \frac12 \iint q^\e(x,y)\bigl[\log \rho_0(x)+ \log \rho^\e(y)\bigr]\, dxdy\notag\\
&\qquad {}+\frac2{\e^2} \iint q^\e(x,y) \bigl[\ph_\e(x)+\ph_\e^*(y)-xy\bigr]\, dxdy\notag\\
&\stackrel{\pref{equiv:dq}}=E(q^\e) - \frac12 E(\rho_0) - \frac12 E(\rho^\e) 
  + \frac1{\e^2} \bigl[d(q^\e)^2 - d(\rho_0,\rho^\e)^2\bigr]
  + \log Z_\e + \frac12 \log \e^2\pi\notag\\
&= J_\e(\rho^\e\;;\rho_0) - \frac1{\e^2}d(\rho_0,\rho^\e)^2 - \frac12 E(\rho^\e)
+ \frac12 E(\rho_0) + \log Z_\e .
\label{ineq:FZI}
\end{align}
The lower-bound estimate
\[
\liminf_{\e\to0} J_\e(\rho^\e\,;\rho_0) - \frac1{\e^2}d(\rho_0,\rho^\e)^2
\geq \frac12 E(\rho)
- \frac12 E(\rho_0)
\]
then follows from the Lemma below, which is proved in Section~\ref{sec:lowerbound}.
\begin{lemma}
\label{lemma:lowerbound}
We have
\begin{enumerate}
\item \label{item:lemma:lowerbound:lsc}
$\liminf_{\e\to0} E(\rho^\e)\geq E(\rho) $;
\item $\limsup_{\e\to0} Z_\e\leq 1$.
\end{enumerate}
\end{lemma}

\medskip

For the \emph{recovery sequence}, part~\ref{item:th:main:recovery} of Theorem~\ref{th:main}, we first define the functional $G_\e\colon \M_1(\R\times\R)\to\R$ by
\[
G_\e(q) := H(q|(\pi_0 q)P^\e) - \frac1{\e^2} d(\pi_0 q,\pi_1 q)^2.
\]
Note that by~\eqref{expansion:H} and~\eqref{equiv:dq}, for any $q$ such that $\pi_0 q = \rho_0$ we have
\begin{multline}
\label{def:duality1}
G_\e(q) = E(q) - E(\rho_0) + \frac 12 \log \e^2\pi\\
+\inf_{\ph} \left\{ \frac2{\e^2}\iint q(x,y) \bigl(\ph(x)+\ph^*(y)-xy\bigr)\, dxdy :\ \ph\text{ convex}\right\}. 
\end{multline}

Now choose for $\ph$ the optimal convex function in the definition of $d(\rho_0,\rho)$, and let the function $\tilde q^\e$ be given by~\pref{def:tildeq}, where $\rho_1^\e$, $\ph_\e$, and $\ph_\e^*$ are replaced by the fixed functions $\rho$, $\ph$, and $\ph^*$. 
Define the correction factor $\chi_\e\in L^1(\pi_0\tilde q^\e)$ by
the condition
\begin{equation}
\label{def:chi_e}
\rho_0(x) = \chi_\e(x) \pi_0\tilde q^\e(x).
\end{equation}
We then set
\begin{align}
q^\e(x,y) &= \chi_\e(x)\tilde q^\e(x,y)\notag\\
&= Z_\e^{-1} \frac1{\e\sqrt\pi}\chi_\e(x)\sqrt{\rho_0(x)}\sqrt{\rho_1(y)}
  \exp\Bigl[\frac2{\e^2}(xy-\ph(x) - \ph^*(y))\Bigr],
\label{def:qe}
\end{align}
so that the first marginal $\pi_0q^\e$ equals $\rho_0$; in Lemma~\ref{lemma:upperbound} below we show that the second marginal converges to $\rho$. Note that the normalization constant $Z_\e$ above is the same as for $\tilde q^\e$, i.e., 
\[
Z_\e =  \frac1{\e\sqrt\pi}\int_K\int_K \sqrt{\rho_0(x)}\sqrt{\rho_1(y)}
  \exp\Bigl[\frac2{\e^2}(xy-\ph(x) - \ph^*(y))\Bigr]\, dxdy.
\]

Since the functions $\ph$ and $\ph^*$ are admissible for $d(\pi_0q^\e,\pi_1q^\e)$, we find with~\eqref{def:phph*}
\begin{align*}
d(\pi_0q^\e,\pi_1q^\e) &\geq \int_{\R} (x^2 - 2\ph(x))\pi_0q^\e(x)\, dx 
 + \int_{\R} (y^2-2\ph^*(y))\,\pi_1q^\e(y)\, dy\\
&= \iint \bigl[x^2 - 2\ph(x)-2\ph^*(y)+ y^2\bigr]q^\e(x,y)\, dxdy.
\end{align*}
Then
\begin{align*}
G_\e(q^\e) &\leq E(q^\e)  - E(\rho_0) + \frac 12 \log \e^2\pi 
+ \frac2{\e^2}\iint q^\e(x,y) \bigl(\ph(x)+\ph^*(y)-xy\bigr) \, dxdy\\
&= -\log Z_\e + \iint q^\e(x,y) \log\chi_\e(x)\, dxdy  \\
&\qquad{} +\frac12 \iint q^\e(x,y) \log \rho_1(y)\, dxdy
 - \frac12 \iint q^\e(x,y)\log \rho_0(x)\, dxdy\\
&= -\log Z_\e + \int \rho_0(x) \log\chi_\e(x)\, dx \\
&\qquad{} +\frac12 \int \pi_1 q^\e(y) \log \rho_1(y)\, dy
 - \frac12 \int \rho_0(x)\log \rho_0(x)\, dx.
\end{align*}
The property~\pref{th:main:recovery} then follows from the lower bound and Lemma below, which is proved in Section~\ref{sec:upperbound}.
\begin{lemma}
\label{lemma:upperbound}
We have
\begin{enumerate}
\item $\lim_{\e\to0} Z_\e = 1$;
\item \label{lemma:upperbound:boundedness}
$\pi_{0,1}\tilde q^\e$ and $\chi_\e$ are bounded on $(0,L)$ from above and away from zero, uniformly in $\e$;
\item $\chi_\e \to 1$ in $L^1(0,L)$;
\item $\pi_1q^\e\to \rho_1$ in $L^1(0,L)$.
\end{enumerate}
\end{lemma}

\section{Upper bound}
\label{sec:upperbound}

In this section we prove Lemma~\ref{lemma:upperbound}, and we place ourselves in the context of the recovery property, part~\ref{item:th:main:recovery}, of Theorem~\ref{th:main}. Therefore we are given $\rho_0,\rho_1\in A_\delta$ with $\rho_0\in C([0,L])$, and as described in Section~\ref{subsec:overview} we have constructed the pair measures $q^\e$ and $\tilde q^\e$ as in~\pref{def:qe}; the convex function $\ph$ is associated with $d(\rho_0,\rho_1)$. The parameter $\delta$  will be determined in the proof of the lower bound; for the upper bound it is sufficient that $0<\delta<1/2$, and therefore that $1/2\leq \rho_0,\rho_1\leq 3/2$. Note that this implies that $\ph''$ and ${\ph^*}''$ are bounded between $1/3$ and $3$.

\medskip
By Aleksandrov's theorem~\cite[Th.~6.4.I]{EvansGariepy92} the convex function $\ph^*$ is twice differentiable at Lebesgue-almost every point $y\in\R$.  Let $N_x\subset \R$ be the set where $\ph$ is not differentiable; this is a Lebesgue null set. Let $N_y\subset \R$ be the set at which $\ph^*$ is not twice differentiable, or at which $({\ph^*})''$ does exist but vanishes; the first set of points is a Lebesgue null set, and the second is a $\rho_1$-null set by~\pref{eq:phpptwo}; therefore $\rho_1(N_y) =0$. Now set 
\[
N = N_x \cup \partial \ph^*(N_y);
\]
here $\partial \ph^*$ is the (multi-valued) sub-differential of $\ph^*$. Then $\rho_0(N) \leq \rho_0(N_x) + \rho_0(\partial \ph^*(N_y)) = 0+\rho_0(\partial \ph^*(N_y)) = \rho_1(N_y) = 0$, where the second identity follows from~\cite[Lemma~4.1]{McCann97}.

Then, since ${\ph^*}'(\ph'(x))=x$, we have for any $x\in \R\setminus N$, 
\[
\ph^*(y) = \ph^*(\ph'(x))+ x(y-\ph'(x)) 
  + \frac12 {\ph^*}''(\ph'(x))(y-\ph'(x))^2 + o((y-\ph'(x))^2),
\]
so that, using $\ph(x) + \ph^*(\ph'(x)) = x\ph'(x)$, 
\[
\ph(x) + \ph^*(y) -xy= 
   \frac12 {\ph^*}''(\ph'(x))(y-\ph'(x))^2 + o((y-\ph'(x))^2).
\]
Therefore for each $x\in \R\setminus N$ the single integral
\begin{multline*}
\frac1\e \int_\R \sqrt{\rho_1(y)} \exp\Bigl[\frac2{\e^2}(xy-\ph(x) - \ph^*(y))\Bigr]\, dy
=\\
=\frac1\e \int_\R \sqrt{\rho_1(y) }\exp\Bigl[-\frac1{\e^2}{\ph^*}''(\ph'(x))(y-\ph'(x))^2 + o(\e^{-2}(y-\ph'(x))^2)\Bigr]\, dy
\end{multline*}
can be shown by Watson's Lemma to converge to 
\begin{equation}
\label{conv:Watson}
\sqrt{\rho_1(\ph'(x))}\sqrt\pi \frac1{\sqrt{{\ph^*}''(\ph'(x))}} = \sqrt\pi \, \sqrt{\rho_0(x)}.
\end{equation}
By Fatou's Lemma, therefore,
\begin{equation}
\label{conv:Ze-upperbound-liminf}
\liminf_{\e\to0} Z_\e \geq 1.
\end{equation}

By the same argument as above, and using the lower bound $\ph''\geq1/3$, we find that 
\begin{equation}
\label{bounds:on-ph''}
xy - \ph(x)-\ph^*(y)\leq \min\left\{ -\frac 16 (x-{\ph^*}'(y))^2, -\frac 16(y-\ph'(x))^2\right\}.
\end{equation}
Then we can estimate
\begin{multline}
\frac1\e \int_\R \int_\R \sqrt{\rho_0(x)}\sqrt{\rho_1(y)}
  \exp\Bigl[\frac2{\e^2}(xy-\ph(x) - \ph^*(y))\Bigr]\, dxdy\\
\leq \frac1\e \int_0^L \int_0^L \sqrt{\rho_0({\ph^*}'(y))}\sqrt{\rho_1(y)}
  \exp\Bigl[\frac2{\e^2}(xy-\ph(x) - \ph^*(y))\Bigr]\, dxdy\\
  + \frac1\e \int_0^L \int_0^L 
    \Bigl|\sqrt{\rho_0(x)}- \sqrt{\rho_0({\ph^*}'(y))}\Bigr|\sqrt{\rho_1(y)}
  \exp\Bigl[-\frac 1{3\e^2}(x-{\ph^*}'(y))^2\Bigr]\, dxdy.
  \label{est:Ze-upperbound-split}
\end{multline}
By the same argument as above, in the first term the inner integral converges at $\rho_1$-almost every $y$ to $\rho_1(y)\sqrt{\pi}$ and is bounded by
\[
\frac1\e \|\rho_0\|_\infty^{1/2} \|\rho_1\|_\infty^{1/2} 
  \int_\R\exp\Bigl[-\frac 1{3\e^2}(x-{\ph^*}'(y))^2\Bigr]\, dx
  = \|\rho_0\|_\infty^{1/2} \|\rho_1\|_\infty^{1/2} \sqrt{3\pi},
\]
so that
\begin{equation}
\label{conv:Ze-upperbound}
\lim_{\e\to0} \frac1\e \int_0^L \int_0^L \sqrt{\rho_0({\ph^*}'(y))}\sqrt{\rho_1(y)}
  \exp\Bigl[\frac2{\e^2}(xy-\ph(x) - \ph^*(y))\Bigr]\, dxdy  = \sqrt\pi.
\end{equation}
To estimate the second term we note that since ${\ph^*}'$ maps $[0,L]$ to $[0,L]$, we can estimate
\[
\Bigl|\sqrt{\rho_0(x)}- \sqrt{\rho_0({\ph^*}'(y))}\Bigr| \leq \omega_{\sqrt{\rho_0}}(|x-{\ph^*}'(y)|),
\qquad\text{for all }(x,y)\in [0,L]\times[0,L],
\]
where $\omega_{\sqrt{\rho_0}}$ is the modulus of continuity of $\sqrt{\rho_0}\in C([0,L])$.
Then
\begin{align}
\frac1\e &\int_0^L \int_0^L 
    \Bigl|\sqrt{\rho_0(x)}- \sqrt{\rho_0({\ph^*}'(y))}\Bigr|\sqrt{\rho_1(y)}
  \exp\Bigl[-\frac 1{3\e^2}(x-{\ph^*}'(y))^2\Bigr]\, dxdy\notag\\
&\leq \frac1\e \omega_{\sqrt{\rho_0}}(\eta)\|\rho_1\|_\infty^{1/2}
  \int_0^L \int_{\{x\in [0,L]: |x-{\ph^*}'(y)|\leq \eta\}}
  \exp\Bigl[-\frac 1{3\e^2}(x-{\ph^*}'(y))^2\Bigr]\, dxdy\notag\\
&\qquad{}+\frac1\e \|\rho_0\|_\infty^{1/2}\|\rho_1\|_\infty^{1/2}
  \int_0^L \int_{\{x\in [0,L]: |x-{\ph^*}'(y)|>\eta\}}
    \exp\Bigl[-\frac 1{3\e^2}(x-{\ph^*}'(y))^2\Bigr]\, dxdy\notag\\
&\leq  \omega_{\sqrt{\rho_0}}(\eta)\|\rho_1\|_\infty^{1/2}
  L \sqrt{3\pi }+\frac1\e \|\rho_0\|_\infty^{1/2}\|\rho_1\|_\infty^{1/2}
  L^2 \exp \Bigl[-\frac {\eta}{3\e^2}\Bigr]. 
\label{est:continuity}
\end{align}
The first term above can be made arbitrarily small by choosing $\eta>0$ small, and for any fixed $\eta>0$ the second converges to zero as $\e\to0$. Combining~\pref{conv:Ze-upperbound-liminf}, \pref{est:Ze-upperbound-split}, \pref{conv:Ze-upperbound}  and~\pref{est:continuity}, we find the first part of Lemma~\ref{lemma:upperbound}: 
\[
\lim_{\e\to0}  Z_\e= 1.
\]

Continuing with part~\ref{lemma:upperbound:boundedness} of Lemma~\ref{lemma:upperbound}, we note that by~\pref{bounds:on-ph''}, e.g.,
\begin{align*}
\pi_0\tilde q^\e(x)
&\leq Z_\e^{-1} \frac1{\e\sqrt \pi} \sqrt{\rho_0(x)}
\int_0^L \sqrt{\rho_1(y)} \exp\Bigl[-\frac 1{3\e^2}(y-\ph'(x))^2\Bigr]\, dy\\
&\leq Z_\e^{-1} \|\rho_0\|_\infty^{1/2}\|\rho_1\|_\infty^{1/2}{\sqrt 3}.
\end{align*}
Since $Z_\e\to1$, $\pi_0\tilde q^\e$ is uniformly bounded from above. A similar argument holds for the upper bound on $\pi_1\tilde q^\e$, and by applying upper bounds on $\ph''$ and ${\ph^*}''$ we also obtain uniform lower bounds on $\pi_0\tilde q^\e$ and $\pi_1\tilde q^e$. The boundedness of $\chi_\e$ then follows from~\pref{def:chi_e} and the bounds on $\rho_0$.

\medskip
We conclude with the convergence of the $\chi_\e$ and $\pi_1 q^\e$. By~\pref{conv:Watson} and~\pref{conv:Ze-upperbound} we have for almost all $x\in (0,L)$,
\[
\pi_0\tilde q^\e(x) = Z_\e^{-1}\sqrt{\rho_0(x)}\frac1{\e\sqrt\pi}\int\sqrt{\rho_1(y)} \exp\Bigl[\frac2{\e^2}(xy-\ph(x)-\ph^*(y))\Bigr]\, dy\longrightarrow \rho_0(x),
\]
and the uniform bounds on $\pi_0\tilde q^\e$ imply that $\pi_0\tilde q^e$ converges to $\rho_0$ in $L^1(0,L)$.
Therefore also $\chi_\e\to1$ in $L^1(0,L)$. A similar calculation gives $\pi_1 q^\e\to\rho_1$ in $L^1(0,L)$. This concludes the proof of Lemma~\ref{lemma:upperbound}.
\qed

\section{Lower bound}
\label{sec:lowerbound}

This section gives the proof of the lower-bound estimate, part~\ref{item:th:main:lowerbound} of Theorem~\ref{th:main}. Recall that in the context of part~\ref{item:th:main:lowerbound} of Theorem~\ref{th:main}, we are given a fixed $\rho_0\in A_\delta\cap C([0,L])$ and a sequence $(\rho^\e)\subset A_\delta$ with $\rho^\e\weakto \rho$. In Section~\ref{subsec:overview} we described how the lower  the lower-bound inequality~\eqref{th:main:lowerbound} follows from two inequalities (see Lemma~\ref{lemma:lowerbound}). The first of these, $\liminf_{\e\to0} E(\rho^\e)\geq E(\rho)$, follows directly from the convexity of the functional~$E$.

The rest of this section is therefore devoted to the proof of the second inequality of Lemma~\ref{lemma:lowerbound}, 
\begin{equation}
\label{ineq:Ze}
\limsup_{\e\to0} Z_\e\leq 1.
\end{equation} 
Here $Z_\e$ is defined in~\pref{def:Ze} as
\[
Z_\e := \frac1{\e\sqrt\pi}\iint\limits_{\R\times\R} \sqrt{\rho_0(x)}\sqrt{\rho^\e(y)}
  \exp\Bigl[\frac2{\e^2}(xy-\ph_\e(x) - \ph_\e^*(y))\Bigr]\, dxdy,
\]
where we extend $\rho_0$ and $\rho^\e$ by zero outside of $[0,L]$, and $\ph_\e$ is associated with $d(\rho_0,\rho^\e)$ as in Lemma~\ref{lemma:phph}. This implies among other things that $\ph_\e$ is twice differentiable on $[0,L]$, and 
\begin{equation}
\label{eq:phpp2}
\ph_\e''(x) = \frac{\rho_0(x)}{\rho^\e(\ph_\e'(x))}
\qquad\text{for all }x\in[0,L].
\end{equation}

We restrict ourselves to the case $L=1$, that is, to the interval $K:=[0,1]$; by a rescaling argument this entails no loss of generality.
We will prove below that there exists a $0<\delta\leq1/3$ such that whenever 
\[
\hat\delta := \max\left\{\|\rho_0-1\|_{L^\infty(K)}, \ \sup_\e \left\|\frac{\rho^\e}{\rho_0}-1\right\|_{L^\infty(K)}\right\}\leq \delta,
\]
the inequality~\eqref{ineq:Ze} holds. This implies the assertion of Lemma~\ref{lemma:lowerbound} and concludes the proof of Theorem~\ref{th:main}.

\medskip

\subsection{Main steps}
A central step in the proof is a reformulation of the integral defining $Z_\e$ in terms of a convolution. Upon writing $y=\ph'_\e(\xi)$ and $x=\xi + \e z$, and using 
$\ph_\e(\xi) + \ph_\e^*(\ph_\e'(\xi)) = \xi \ph_\e'(\xi)$, we can rewrite the exponent in $Z_\e$ as
\begin{align}
\label{calc:1}
\ph_\e(x) &+ \ph_\e^*(y) -xy = \\
&=\ph_\e(\xi+\e z) + \ph_\e^* (\ph_\e'(\xi)) - (\xi+\e z) \ph_\e'(\xi)\label{calc:2}\\
&=\ph_\e(\xi+\e z) - \ph_\e(\xi) - \e z \ph_\e'(\xi)\notag\\
&=\e^2 \int_0^z (z-s)\ph_\e''(\xi+\e s)\, ds\notag\\
&=\frac{z^2\e^2}2 \bigl(\kappa_\e^z*\ph_\e''\bigr)(\xi),
\label{calc:5}
\end{align} 
where we define the convolution kernel $\kappa_\e^z$ by
\[
\kappa^{z}_\e(s) = \e^{-1}\kappa^z(\e^{-1}s)
\qquad\text{and}\qquad
\kappa^z(\sigma) = 
\begin{cases}
\frac2{z^2}(z+\sigma) & \text{if } -z\leq\sigma \leq 0\\
-\frac2{z^2}(z+\sigma) & \text{if } 0\leq \sigma\leq -z\\
0 &\text{otherwise.}
\end{cases}
\]

\begin{figure}[h]
\centering
\psfig{figure=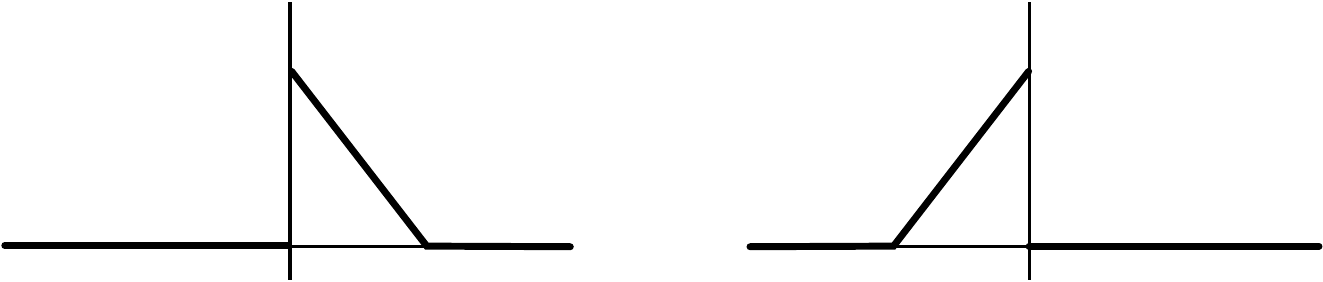,height=3cm}
\caption{The function $\kappa^z_\e$ for negative and positive values of $z$.}
\end{figure}

While the domain of definition of~\eqref{calc:1} is a convenient rectangle $K^2=[0,1]^2$, after transforming to~\eqref{calc:2}  this domain becomes an inconvenient $\e$-dependent parallellogram in terms of $z$ and $\xi$. The following Lemma therefore allows us to switch to a more convenient setting, in which we work on the flat torus $\T=\R/\Z$ (for $\xi$) and $\R$ (for $z$).

\begin{lemma}
\label{lemma:approx_u}
Set $u\in L^\infty(\T)$ to be the periodic function on the torus $\T$ such that $u(\xi) = \ph_\e''(\xi)$  for all $\xi\in K$ (in particular, $u\geq0$).
There exists a function $\omega\in C([0,\infty))$ with $\omega(0)=0$, depending only on $\rho_0$, such that for all $\hat \delta\leq1/3$
\[
\sqrt \pi\; Z_\e \leq \omega(\e) + \int_\T \rho_0(\xi)\sqrt{u(\xi)}
  \int_\R \exp[-(\kappa^z_\e*u)(\xi)z^2]\, dzd\xi.
\]
\end{lemma}

Given this Lemma it is sufficient to estimate the integral above. To explain the main argument that leads to the inequality~\eqref{ineq:Ze}, we give a heuristic description that is mathematically false but morally correct; this will be remedied below. 

We  approximate in $Z_\e$  an expression of the form $e^{-a-b}$ by $e^{-a}(1-b)$ (let us call this perturbation 1), and we set $\rho_0\equiv 1$ (perturbation 2). Then
\begin{align*}
\sqrt \pi\; Z_\e &- \omega(\e) \leq \int_\T \sqrt{u(\xi)}
  \int_\R e^{-u(\xi)z^2}\bigl[1-(\kappa^z_\e*u - u)(\xi)z^2\bigr]\, dzd\xi\\
&= \int_\T \sqrt{u(\xi)}
  \int_\R e^{-u(\xi)z^2}\, dzd\xi
  -\int_\T \sqrt{u(\xi)}
  \int_\R e^{-u(\xi)z^2}\bigl[(\kappa^z_\e*u) - u\bigr](\xi)z^2\, dzd\xi.
\end{align*}
The first term can be calculated by setting $\zeta=z\sqrt{u(\xi)}$, 
\[
\int_\T \int_\R e^{-\zeta^2}\, d\zeta d\xi = \int_\T \sqrt \pi \, d\xi = \sqrt\pi.
\]
In the second term we approximate $(\kappa^z_\e*u)(\xi) - u(\xi)$ by $c u''(\xi)\e^2 z^2$, where $c = \frac14 \int s^2 \kappa^z(s)\, ds$ (this is perturbation 3). Then this term becomes, using the same transformation to $\zeta$ as above,
\begin{align}
\notag
-c\e^2\int_\T \sqrt{u(\xi)}  \int_\R e^{-u(\xi)z^2}u''(\xi)z^4\, dzd\xi
  &= -c\e^2\int_\T \frac{u''(\xi)}{u(\xi)^2} \int_\R e^{-\zeta^2}\zeta^4\, d\zeta d\xi\\
&= -2c\e^2 \int_\T \frac{u'(\xi)^2}{u(\xi)^3} \,\sqrt\pi\, d\xi.
\label{perturbation_argument}
\end{align}
Therefore this term is negative and of order $\e^2$ as $\e\to0$, and the inequality~\eqref{ineq:Ze} follows.

The full argument below is based on this principle, but corrects for the three perturbations made above. Note that the difference 
\begin{equation}
\label{perturbation_2}
e^{-a-b} - e^{-a}(1-b)
\end{equation}
is positive, so that the ensuing correction competes with~\eqref{perturbation_argument}. In addition,  both the beneficial contribution from~\eqref{perturbation_argument} and the detrimental contribution from~\eqref{perturbation_2} are of order $\e^2$. The argument only works because the corresponding constants happen to be ordered in the right way, and then only when  $\|u-1\|_\infty$ is small. This is the reason for the restriction represented by $\delta$.

\subsection{Proof of Lemma~\ref{lemma:approx_u}}

Since $\hat \delta\leq 1/3$, then~\pref{eq:phpp2} implies that $\ph'_\e$ is Lipschitz on $K$, 
and we can transform $Z_\e$  following the sequence~(\ref{calc:1})--(\ref{calc:5}), and using $\supp\rho_0, \rho^\e = K$:
\begin{align*}
\sqrt \pi\; Z_\e &= \frac1{\e}\int_K \sqrt{\rho^\e(y)}\int_K \sqrt{\rho_0(x)}
  \exp\Bigl[\frac2{\e^2}(xy-\ph_\e(x) - \ph_\e^*(y))\Bigr]\, dxdy\\
&=\int_K \sqrt{\rho^\e(\ph_\e'(\xi))}
  \int\limits_{-\xi/\e}^{(1-\xi)/\e} \sqrt{\rho_0(\e z+{\ph_\e^*}'(y))}\exp[-(\kappa_\e^z*\ph_\e'')(\xi)z^2]\, dz\,\ph_\e''(\xi)d\xi\\
&= \int_K \sqrt{\rho_0(\xi)}\sqrt{\ph_\e''(\xi)}
  \int_\R\sqrt{\rho_0(\xi+ \e z)}\exp[-(\kappa_\e^z*\ph_\e'')(\xi)z^2]\, dzd\xi,
\end{align*}
where we used~\eqref{eq:phpp2} in the last line.

Note that $(\kappa_\e^z*\ph_\e'')(\xi)z^2=(\kappa_\e^z * u)(\xi)z^2$ for all $z\in\R$ and for all $\xi\in K^{\e z}$, where $K^{\e z}$ is the interval $K$ from which an interval of length $\e z$ has been removed from the left (if $z<0$) or from the right (if $z>0$). Therefore
\begin{align*}
\sqrt\pi Z_\e &- \int_\T\rho_0(\xi)\sqrt{u(\xi)} \int_\R \exp[-(\kappa_\e^z*u)(\xi)z^2]\, dzd\xi\\
=& \int_\R\int_{K^{\e z}} \sqrt{\rho_0(\xi)}\sqrt{u(\xi)}
  \Bigl(\sqrt{\rho_0(\xi+\e z)}-\sqrt{\rho_0(\xi)}\Bigr)\exp[-(\kappa_\e^z*u)(\xi)z^2]\, d\xi dz\\
&+ \int_\R\int_{K\setminus K^{\e z}}\sqrt{\rho_0(\xi)}\sqrt{u(\xi)}
  \sqrt{\rho_0(\xi+\e z)}\exp[-(\kappa_\e^z*u)(\xi)z^2]\, d\xi dz\\
&- \int_\R\int_{K\setminus K^{\e z}} \rho_0(\xi)\sqrt{u(\xi)}  
  \exp[-(\kappa_\e^z*u)(\xi)z^2]\, d\xi dz.
\end{align*}
The final term is negative and we discard it. From the assumption $\hat \delta\leq 1/2$ we deduce $\|u-1\|_\infty\leq 1/2$, so that the first term on the right-hand side can be estimated from above (in terms of the modulus of continuity $\omega_{\rho_0}$ of $\rho_0$) by 
\[
\|\rho_0\|_{L^\infty(K)}^{1/2}\|u\|_{L^\infty(K)}^{1/2} \int_\R \int_{K^{\e z}} \omega_{\rho_0}(\e z) e^{-z^2/2}\, d\xi dz
\leq\frac32  \int_\R \omega_{\rho_0}(\e z) e^{-z^2/2}\, dz,
\]
which converges to zero as $\e\to0$, with a rate of convergence that depends only on~$\rho_0$. 
Similarly, the middle term we estimate by
\[
\|\rho_0\|_{L^\infty(K)}\|u\|_{L^\infty(K)}^{1/2} \int_\R |{K\setminus K^{\e z}}| e^{-z^2/2}\, dz
\leq \Bigl(\frac32\Bigr)^{3/2} \e \int_\R |z| e^{-z^2/2}\, dz,
\]
which converges to zero as $\e\to0$.
\qed

\subsection{The semi-norm $\|\cdot\|_\e$}
\label{subsec:Xe}

It is convenient to introduce a specific semi-norm for the estimates that we make below, which takes into account the nature of the convolution expressions. 
On the torus $\T$ we define 
\[
\|u\|_{\e}^2  := \sum_{k\in \Z} |u_k|^2 \bigl(1-e^{-\pi^2k^2\e^2}\bigr),
\]
where the $u_k$ are the Fourier coefficients of $u$, 
\[
u(x) = \sum_{k\in\Z} u_k e^{2\pi i k x}.
\]
The following Lemmas give the relevant properties of this seminorm. 

\begin{lemma}
\label{lemma:FD}
For $\e>0$,
\begin{equation}
\label{id:X_e}
\int_\R  e^{-z^2} \int_\T (u(x+\e z)-u(x))^2 \, dx dz
= 2\sqrt \pi \| u\|_{\e}^2.
\end{equation}
\end{lemma}

\begin{lemma}
\label{lemma:uksq}
For $\e>0$,
\begin{equation}
\label{ineq:lemma7}
\int_\R\int_\T e^{-z^2}(u(x)-\kappa_\e^z*u(x))^2 z^4\, dxdz
\leq \frac 56 \, \sqrt\pi \, \|u\|_{{\e}}^2.
\end{equation}
\end{lemma}
\begin{lemma}
\label{lemma:Xee}
For $\alpha>0$ and $\e>0$,
\begin{equation}
\label{ineq:lemma8}
\|u\|_{{\e/\alpha}}\leq
\begin{cases}
\|u\|_{\e} & \text{if }\alpha\geq1\\
\frac1\alpha \|u\|_{\e} & \text{if }0<\alpha\leq 1,
\end{cases}
\end{equation}
where $\|\cdot\|_{\e/\alpha}$ should be interpreted as $\|\cdot\|_\e$ with $\e$ replaced by $\e/\alpha$.
\end{lemma}

The proofs of these results are given in the appendix.

\subsection{Conclusion}
To alleviate notation we drop the caret from $\hat\delta$ and simply write $\delta$.
Following the discussion above  we estimate 
\begin{multline}
\int_\T 
\rho_0(\xi)\sqrt{u(\xi)}
  \int_\R \exp[-(\kappa^z_\e*u)(\xi)z^2]\, dzd\xi
  = \int_\T\int_\R \rho_0(\xi)\sqrt{u(\xi)} e^{-u(\xi)z^2}\, dzd\xi\\
+ \int_\T\int_\R \rho_0(\xi)\sqrt{u(\xi)} e^{-u(\xi)z^2}[u(\xi)-\kappa_\e^z*u(\xi)]z^2\, dzd\xi
+ R,
\label{split:I}
\end{multline}
where
\begin{align*}
R &= \int_\T \int_\R \rho_0(\xi) \sqrt{u(\xi)} e^{-u(\xi)z^2}
  \Bigl[\exp[(u(\xi)-\kappa^z_\e*u(\xi))z^2] - 1 - (u(\xi)-\kappa_\e^z*u(\xi))z^2\Bigr]\, dzd\xi\\
&\leq (1+\delta)^{3/2}\int_\T \int_\R  e^{-u(\xi)z^2}
  \Bigl[\exp[(u(\xi)-\kappa^z_\e*u(\xi))z^2] - 1 - (u(\xi)-\kappa_\e^z*u(\xi))z^2\Bigr]\, dzd\xi.
\end{align*}
Since $\|u-1\|_{L^\infty(\T)}\leq \delta$,  we have $\|u-\kappa_\e^z*u\|_{L^\infty(\T)}\leq 2\delta$ and therefore
\[
\exp[(u(\xi)-\kappa^z_\e*u(\xi))z^2] - 1 - (u(\xi)-\kappa_\e^z*u(\xi))z^2
\leq \frac12 e^{2\delta z^2} (u(\xi)-\kappa^z_\e*u(\xi))^2z^4, 
\]
so that
\begin{align*}
R&\leq \frac{(1+\delta)^{3/2}}2\int_\T \int_\R  e^{(-u(\xi)+2\delta)z^2}(u(\xi)-\kappa^z_\e*u(\xi))^2z^4
\, dzd\xi\\
&\leq \frac{(1+\delta)^{3/2}}2\int_\T \int_\R  e^{(-1+3\delta)z^2}(u(\xi)-\kappa^z_\e*u(\xi))^2z^4
\, dzd\xi.
\end{align*}
Setting $\alpha= \sqrt{1-3\delta}$ and $\zeta = \alpha z$, we find
\[
R\leq \frac{(1+\delta)^{3/2}}{2(1-3\delta)^{5/2}}\int_\T \int_\R e^{-\zeta^2}(u(\xi)-\kappa^{\zeta/\alpha}_\e*u(\xi))^2\zeta^4\, d\zeta d\xi.
\]
Noting that $\kappa_\e^{\zeta/\alpha} = \kappa_{\e/\alpha}^\zeta$, we have with $\tilde \e : =\e/\alpha = \e(1-3\delta)^{-1/2}$
\begin{align}
R &\leq \frac{(1+\delta)^{3/2}}{2(1-3\delta)^{5/2}}\int_\T \int_\R e^{-\zeta^2}(u(\xi)-\kappa^{\zeta}_{\tilde \e}*u(\xi))^2\zeta^4\, d\zeta d\xi\notag\\
&\stackrel{\pref{ineq:lemma7}}\leq \frac{(1+\delta)^{3/2}}{2(1-3\delta)^{5/2}}\, \frac 56 \, \sqrt\pi \, \|u\|_{{\tilde \e}}^2\notag\\
&\stackrel{\pref{ineq:lemma8}}\leq \frac{(1+\delta)^{3/2}}{2(1-3\delta)^{7/2}}\, \frac 56 \, \sqrt\pi \, \|u\|_{{\e}}^2.
\label{est:R}
\end{align}

\bigskip

We next calculate
\begin{equation}
\int_\T\int_\R \rho_0(\xi)\sqrt{u(\xi)} e^{-u(\xi)z^2}\, dzd\xi
= \int_\T\rho_0(\xi)\int_\R  e^{-\zeta^2}\, d\zeta d\xi
= \sqrt\pi \int_\T \rho_0(\xi)\, d\xi = \sqrt\pi.
\label{calc:simple_term}
\end{equation}

Finally we turn to the term
\[
I := \int_\T\int_\R \rho_0(\xi)\sqrt{u(\xi)} e^{-u(\xi)z^2}(u(\xi)-\kappa_\e^z*u(\xi))z^2\, dzd\xi.
\]
\begin{lemma}
\label{lemma:linear_term}
Let $\e>0$, let $\rho_0\in L^\infty(\T)\cap C([0,1])$ with $\int_\T \rho_0 = 1$, and let $u\in L^\infty(\T)$. Recall that $0<\delta<1/3$ with 
\[
\|\rho_0-1\|_{L^\infty(\T)}\leq \delta 
\qquad\text{and}\qquad 
\|u-1\|_{L^\infty(\T)}\leq \delta.
\]
Then 
\[
I\leq -\frac12 \frac{1-\delta}{(1+\delta)^2}\sqrt\pi 
\|u\|_{\e}^2 + r_\e,
\]
where 
$r_\e\to0$ uniformly in $\delta$.
\end{lemma}

From this Lemma and the earlier estimates the result follows. Combining Lemma~\ref{lemma:approx_u} with~\pref{split:I}, \pref{calc:simple_term}, Lemma~\ref{lemma:linear_term} and~\pref{est:R}, 
\begin{align*}
\sqrt\pi \, Z_\e &\leq \sqrt\pi 
-\frac12 \frac{1-\delta}{(1+\delta)^2}\sqrt\pi \|u\|_{\e}^2 
+ \frac{(1+\delta)^{3/2}}{(1-3\delta)^{7/2}}\, \frac 5{12} \, \sqrt\pi \, \|u\|_{{\e}}^2
+ S_\e,
\end{align*}
where $S_\e = \omega(\e) + r_\e$ converges to zero as $\e\to0$, uniformly in $\delta$. Since $1/2>5/12$, for sufficiently small $\delta>0$ the two middle terms add up to a negative value. Then it follows that $\limsup_{\e\to0} Z_\e\leq 1$.

\begin{proof}[Proof of Lemma~\ref{lemma:linear_term}]
Writing $I$ as 
\[
I = 2\int_\T \rho_0(\xi)\sqrt{u(\xi)} \int_\R \int_0^ze^{-u(\xi)z^2}
   (z-\sigma) (u(\xi)-u(\xi+\e \sigma))\, d\sigma dzd\xi,
\]
we apply Fubini's Lemma in the $(z,\sigma)$-plane to find
\begin{align*}
I &= -2\int_\T \rho_0(\xi)\sqrt{u(\xi)} \int_0^\infty \int_\sigma^\infty e^{-u(\xi)z^2}
  (z-\sigma) \bigl[u(\xi+\e \sigma)-2u(\xi)+u(\xi-\e \sigma)\bigr]\, dz d\sigma d\xi\\
 &= -2 \int_0^\infty \sigma  \int_\T \rho_0(\xi) \bigl[u(\xi+\e \sigma)-2u(\xi)+u(\xi-\e \sigma)\bigr]h(\sigma^2 u(\xi))\, d\xi d\sigma ,
\end{align*}
where
\begin{equation}
\label{est:h-upperbound}
h(s) := \frac1{\sqrt s} \int_{\sqrt s}^\infty e^{-\zeta^2}(\zeta-\sqrt s)\, d\zeta \;\leq \;\frac1{2\sqrt s} e^{-s}.
\end{equation}
Since $\|u-1\|_\infty \leq \delta$, 
\begin{equation}
\label{ineq:hprime}
h'(\sigma^2 u) = \frac{-1}{4\sigma^3 u^{3/2}} e^{-u\sigma^2} 
  \leq \frac{-1}{4\sigma^3} \, \frac1{(1+\delta)^{3/2}} e^{-(1+\delta)\sigma^2}.
\end{equation}
Then, writing $D_{\e\sigma}f(\xi)$ for $f(\xi+\e\sigma)-f(\xi)$, we have
\begin{align*}
\int_\T &\rho_0(\xi) \bigl[u(\xi+\e \sigma)-2u(\xi)+u(\xi-\e \sigma)\bigr] h(\sigma^2 u(\xi))\, d\xi =\\
&= -\int_\T \rho_0(\xi) D_{\e\sigma}u(\xi) D_{\e\sigma}h(\sigma^2u)(\xi)\, d\xi
 - \int_\T D_{\e\sigma}\rho_0(\xi) D_{\e\sigma}u(\xi) h(\sigma^2u(\xi+\e\sigma))\, d\xi,
\end{align*}
so that
\begin{align*}
I &= 2\int_0^\infty \sigma \int_\T \rho_0(\xi) D_{\e\sigma}u(\xi) D_{\e\sigma}h(\sigma^2u)(\xi)\, d\xi d\sigma
 +2\int_0^\infty \sigma  \int_\T D_{\e\sigma}\rho_0(\xi) D_{\e\sigma}u(\xi) h(\sigma^2u(\xi+\e\sigma))\, d\xi d\sigma\\
 &= I_a +I_b.
\end{align*}
Taking $I_b$ first, we estimate one part of this integral with~\eqref{est:h-upperbound} by
\begin{align*}
2\int_0^\infty &\sigma  \int_0^{1-\e\sigma} D_{\e\sigma}\rho_0(\xi) D_{\e\sigma}u(\xi) h(\sigma^2u(\xi+\e\sigma))\, d\xi d\sigma\\
&\leq 2\int_0^\infty \sigma \omega_{\rho_0}(\e\sigma)\,
  2\delta\, \frac1{2\sigma \sqrt{1-\delta}} e^{-(1-\delta)\sigma^2}\, d\sigma\\
&\leq \frac{2\delta}{\sqrt{1-\delta}} \int_0^\infty \omega_{\rho_0}(\e\sigma)
  e^{-(1-\delta)\sigma^2}\, d\sigma,
\end{align*}
and this converges to zero as $\e\to0$ uniformly in $0<\delta<1/3$. The remainder of $I_b$ we estimate
\begin{align*}
2\int_0^\infty &\sigma  \int_{1-\e\sigma}^1 D_{\e\sigma}\rho_0(\xi) D_{\e\sigma}u(\xi) h(\sigma^2u(\xi+\e\sigma))\, d\xi d\sigma\\
&\leq 2\int_0^\infty \e\sigma^2 \,2\delta\, \frac1{2\sigma \sqrt{1-\delta}} e^{-(1-\delta)\sigma^2}\, d\sigma\\
&= \frac{2\e\delta}{\sqrt{1-\delta}} \int_0^\infty \sigma
  e^{-(1-\delta)\sigma^2}\, d\sigma,
\end{align*}
which again converges to zero as $\e\to0$, uniformly in $\delta$.

To estimate $I_a$ we note that by~\eqref{ineq:hprime} and the chain rule,
\[
D_{\e\sigma} h(\sigma^2 u)(\xi)\leq -\frac{1}{4\sigma^3} \, \frac1{(1+\delta)^{3/2}} e^{-(1+\delta)\sigma^2}  D_{\e\sigma} u(\xi) \;\sigma^2,
\]
and thus
\begin{align*}
I_a &\stackrel{\pref{ineq:hprime}}\leq
  -\frac{1-\delta}{2(1+\delta)^{3/2}}\int_0^\infty e^{-(1+\delta)\sigma^2} \int_\T 
    (D_{\e\sigma}u(\xi))^2  \, d\xi d\sigma \\
& =-\frac{1-\delta}{2(1+\delta)^{2}}\int_0^\infty e^{-s^2} \int_\T 
    (D_{\e s /\sqrt{1+\delta}}u(\xi))^2  \, d\xi ds\\ 
&\stackrel{\eqref{id:X_e}}= -\frac{1-\delta}{2(1+\delta)^{2}}\sqrt\pi\|u\|^2_{{\e/ \sqrt{1+\delta}}} \\ 
&\stackrel{\eqref{ineq:lemma8}}\leq -\frac{1-\delta}{2(1+\delta)^2}\sqrt\pi\|u\|^2_{{\e}}.
\end{align*}
\end{proof}

\appendix
\section{Proofs of the Lemmas in Section~\ref{subsec:Xe}}

\begin{proof}[Proof of Lemma~\ref{lemma:FD}]
Since the left and right-hand sides are both quadratic in $u$, it is sufficient to prove the lemma for a single Fourier mode $u(x) = \exp 2\pi i k x$, for which
\begin{align*}
\int_\R  e^{-z^2} \int_\T (u(x+\e z)-u(x))^2 \, dx dz
&= \int_\R e^{-z^2} |\exp 2\pi i k \e z-1|^2\, dz\\
&= 2\int_\R e^{-z^2} (1-\cos 2\pi k\e z)\, dz\\
&= 2\sqrt \pi (1-e^{-\pi^2 k^2\e^2}),
\end{align*}
since 
\[
\int_\R e^{-z^2} \, dz = \sqrt \pi \qquad\text{and}\qquad
\int_\R e^{-z^2} \cos \omega z\, dz = \sqrt \pi \;e^{-\omega^2/4}.
\]
\end{proof}

\begin{proof}[Proof of Lemma~\ref{lemma:uksq}]
Again it is sufficient to prove the lemma for a single Fourier mode $u(x) = \exp 2\pi i k x$, for which 
\[
\int_\R\int_\T e^{-z^2}(u(x)-\kappa_\e^z*u(x))^2 z^4\, dxdz
= \int_\R e^{-z^2} z^4 |1-\widehat{\kappa_\e^z}(k)|^2\, dz.
\]
Writing $\omega := 2\pi k \e$, the Fourier transform of $\kappa_\e^z$ on $\T$ is calculated to be 
\[
\widehat{\kappa_\e^z}(k) = \int_0^1 \kappa_\e^z (x) e^{-2\pi ikx}\, dx = -\frac{2}{\omega^2z^2} \bigl[e^{i\omega z} - 1 - i\omega z\bigr].
\]
Then
\[
1-\widehat{\kappa_\e^z}(k) = \frac2{\omega^2z^2} \Bigl[ e^{i\omega z} - 1 - i\omega z +\frac{\omega^2z^2}2\Bigr], 
\]
so that
\begin{align*}
z^4 |1-\widehat{\kappa_\e^z}(k)|^2 &= 
  \frac{4}{\omega^4}\Bigl[ \Bigl(1-\cos\omega z - \frac{\omega^2z^2}2\Bigr)^2 
  + (\sin\omega z - \omega z)^2\Bigr]\\
&=\frac{4}{\omega^4}\Bigl[ 2 - 2\cos\omega z+ \frac{\omega^4z^4}4
-2\omega z\sin\omega z+ \omega^2z^2 \cos\omega z\Bigr]. 
\end{align*}
We then calculate
\begin{align*}
&\int_\R e^{-z^2} z^4\, dz = \frac34 \sqrt \pi\\
&\int_\R e^{-z^2} \cos \omega z\, dz = \sqrt \pi \;e^{-\omega^2/4}\\
&\int_\R e^{-z^2} z\sin \omega z \, dz = \frac\omega 2\sqrt \pi\;e^{-\omega^2/4}\\
&\int_\R e^{-z^2} z^2 \cos \omega z\, dz = \sqrt \pi \;e^{-\omega^2/4} \Bigl(\frac12 - \frac{\omega^2}4\Bigr)
\end{align*}
implying that 
\begin{align*}
\int_\R e^{-z^2} z^4 |1-\widehat{\kappa_\e^z}(k)|^2\, dz
&= \frac{4\sqrt\pi}{\omega^4} \left[ 2 - 2e^{-\omega^2/4} + \frac3{16}\omega^4 
-\omega^2 e^{-\omega^2/4} + \omega^2e^{-\omega^2/4}\Bigl(\frac12 - \frac{\omega^2}4\Bigr)\right]\\
&= \frac{4\sqrt\pi}{\omega^4} \left[ 2 - 2e^{-\omega^2/4} + \frac3{16}\omega^4 
-\frac12{\omega^2} e^{-\omega^2/4} - \frac14\omega^4e^{-\omega^2/4} \right].
\end{align*}
We conclude the lemma by showing that the right-hand side is bounded from above by
\[
\frac56\sqrt\pi (1-e^{-\omega^2/4}).
\]
Indeed, subtracting the two we find
\[
\frac{4\sqrt\pi}{\omega^4} \left[ 2 - 2e^{-\omega^2/4} + \frac3{16}\omega^4 
-\frac12{\omega^2} e^{-\omega^2/4} - \frac14\omega^4e^{-\omega^2/4} - \frac5{24}\omega^4(1-e^{-\omega^2/4})\right],
\]
and setting $s:= \omega^2/4$ the sign of this expression is determined by
\[
2(1-e^{-s}) -\frac13s^2 - 2se^{-s} -\frac23s^2e^{-s}.
\]
This function is zero at $s=0$, and its derivative is
\[
-\frac23 s +\frac23 se^{-s} + \frac23 s^2e^{-s}
\]
which is negative for all $s\geq0$ by the inequality $e^{-s}(1+s)\leq 1$.
\end{proof}

\begin{proof}[Proof of Lemma~\ref{lemma:Xee}]
Since the function $\alpha\mapsto 1-e^{-\pi^2k^2\e^2/\alpha^2}$ is decreasing in $\alpha$, the first inequality follows immediately. To prove the second it is sufficient to show that $1-e^{-\beta x}\leq \beta (1-e^{-x})$ for $\beta>1$ and $x>0$, which can be recognized by differentiating both sides of the inequality.
\end{proof}

%


\bibliographystyle{alpha}
\bibliography{ref}

\end{document}